\documentclass[reqno]{amsart}
\usepackage{hyperref}
\usepackage{amsmath}
\usepackage{color}
\usepackage{enumerate}
\usepackage{comment}
\usepackage{mathrsfs}
\usepackage{appendix}
\usepackage{hyperref}
\usepackage{footmisc}

\makeatletter
\@namedef{subjclassname@2020}{%
  \textup{2020} Mathematics Subject Classification}
\makeatother

\begin{document}
\title[Infinite transition solutions]
{Infinite transition solutions for an Allen--Cahn equation}

\author[W.-L. Li]
{Wen-Long Li}

\address{Wen-Long Li \newline
 School of Sciences,
 Hangzhou Dianzi University, Hangzhou 310018, P. R. China}
\email{liwenlongchn@gmail.com}

\date{\today}

\begin{abstract}
  We give another proof of a theorem of Rabinowitz and Stredulinsky 
  obtaining infinite transition solutions for an Allen--Cahn equation.
  Rabinowitz and Stredulinsky 
  have constructed infinite transition solutions as locally minimal solutions, 
  but it is still an interesting question to establish these solutions by other method.
  Our result may attract the interest of constructing solutions
  with the shape of locally minimal solutions 
  of Rabinowitz and Stredulinsky for problems
  defined on descrete group.
\end{abstract}

\maketitle
\numberwithin{equation}{section}
\newtheorem{theorem}{Theorem}[section]
\newtheorem{lemma}[theorem]{Lemma}
\newtheorem{definition}[theorem]{Definition}
\newtheorem{proposition}[theorem]{Proposition}
\newtheorem{remark}[theorem]{Remark}

\newcommand{\R}{\mathbb{R}}
\newcommand{\Z}{\mathbb{Z}}
\newcommand{\N}{\mathbb{N}}
\newcommand{\Q}{\mathbb{Q}}
\newcommand{\ud}{\mathrm{d}}
\newcommand{\T}{\mathbb{T}}
\newcommand{\MMR}{\mathcal{R}}
\newcommand{\MY}{\mathcal{Y}}
\newcommand{\MM}{\mathcal{M}}
\newcommand{\MMM}{\mathscr{M}}

\allowdisplaybreaks

\section{Introduction}\label{sec:1}

In their elegant memoir \cite{RS},
Rabinowitz and Stredulinsky studied the following Allen--Cahn equation:
\begin{equation}\label{eq:PDE}
-\Delta u+F_{u}(x, u)=0, \quad x \in \mathbb{R}^{n}, \tag{PDE}
\end{equation}
where $F$ satisfies
\begin{equation}\label{eq:F1}
  \left\{
    \begin{array}{l}
      F\in C^2(\R^n\times\R,\R), \\
      \textrm{$F$ is $1$-periodic in all of its arguments.}
    \end{array}
  \right.
\end{equation}
\eqref{eq:PDE} belongs to a broad class of elliptic equations,
which were first initiated by Moser \cite{moser}.
The goal of Moser 
was to establish a codimension $1$ version of  
Aubry--Mather theory---an elegant theory for some $1$ dimensional variational 
problems (cf. \cite{Aubry, Mather, DynRep})---for elliptic PDEs.
To avoid introducing more concepts in the introduction,
we shall restrict ourselves to \eqref{eq:PDE}.
For \eqref{eq:PDE},
Moser studied its minimal and without self-intersection solutions.
These solutions exhibit similar properties of minimal solutions in Aubry--Mather theory.
For example,
every minimal and without self-intersection solution $u$ 
has a rotation vecter $\alpha=\alpha(u)\in\R^n$,
and for any $\alpha\in\R^n$ there exists a minimal and without self-intersection solution $u$ 
such that the rotation vector of $u$ is $\alpha$.
Let us denote all the minimal and without self-intersection solutions of \eqref{eq:PDE} by $\MMM$ and 
set $$\MMM_{\alpha}:=\{u\in\MMM \,|\, \textrm{The rotation vector of $u$ is $\alpha$}\}.$$
Bangert \cite{Bangert1, Bangert2, Bangert} studied Moser's problem.
He introduced the second invariants to clarify 
the structure of $\MMM_{\alpha}$.
For instance,
if $\alpha\in\Q^n$,
$\MMM_{\alpha}$ consists of periodic solutions and heteroclinic solutions.
Bangert showed his results by carefully studying the structures of 
the minimal recurrent sets in $\MMM_{\alpha}$.

For $\alpha\in\Q^n$,
Rabinowitz and Stredulinsky established the existence of
heteroclinic solutions in $\MMM_{\alpha}$ 
by minimization method 
(different with Bangert's method),
and then using excellent variational techniques 
they constructed a lot of complicated homoclinic and heteroclinic solutions of \eqref{eq:PDE}.
Part of their results can be generalized to the cases of $\MMM_{\alpha}$ of 
irrational rotation vector $\alpha$,
i.e., $\alpha\in\R^n\setminus \Q^n$.
Please see \cite{LW} for more details.
Although Rabinowitz and Stredulinsky \cite{RS} 
(please see \cite{Bo1, Bo2, Rabi2014} for further studies by Bolotin and Rabinowitz)
only studied \eqref{eq:PDE}
--- a special class of Moser's variational problem,
most of their methods and results also hold for more general Lagrangians 
(please see \cite[Remark 1.9]{Rabi2014}).

As the title of \cite{RS} indicates,
lots of solutions of Rabinowitz and Stredulinsky are locally minimal,
including infinite transition solutions considered in the present paper.
The so-called \emph{transition} means the solution changes its states in spatial direction.
To be more precise,
assume $v_0<w_0\in\MMM_0$ are periodic solutions,
and suppose there does not exist periodic minimal and without self-intersection solution lying between $v_0$ and $w_0$.
Heteroclinic solution mentioned above changes its states in $x_1$ from $v_0$ (resp. $w_0$) to $w_0$ (resp. $v_0$)
and can be seen as $1$ transition solution.
If a solution changes its states in $x_1$ from $v_0$ to $w_0$ and returns back to $v_0$,
it is called a $2$ transition solution.
Similarly, one can define multi- and infinite transition solutions.

Let us 
explain what is the meaning of local minimality property for functions in $W^{1,2}_{loc}(\R\times \T^{n-1})$.
\begin{definition}[{cf. \cite[Proposition 8.12]{RS}}]\label{def:locally_minimal}
For any $x=(x_1,\cdots, x_n)\in\R^n$ and small $r>0$, 
$U\in W^{1,2}_{loc}(\R\times \T^{n-1})$ is said to be locally minimal in $B_r(z)$
if 
$U$ minimizes
$$I_{r,z}(u):=\int_{B_r(z)} \left[ \frac{1}{2}|\nabla u|^2 + F(x,u) \right] \ud x$$
over
$$
E_{r,z}:=\{u\in W^{1,2}_{loc}(\R\times \T^{n-1}) \,|\, u=U \text{ on } (\R\times \T^{n-1}\setminus B_r(z))\}.
$$
\end{definition}
Observe that $U$ is a solution of \eqref{eq:PDE} in $B_r(z)$ if it is locally minimal in the same domain.
In their second method of constructing multitransition solutions \cite[Theorem 8.11]{RS},
Rabinowitz and Stredulinsky 
proved the constrained minimizers are solutions of \eqref{eq:PDE}
by showing these solutions satisfy local minimality property.
Since the multitransition solutions have uniform bounds 
in $C^{2,\alpha}_{loc}(\R\times \T^{n-1})$ for any $\alpha\in (0,1)$,
one obtain infinite transition solutions by an approximation argument.
In the key step, that is, multitransition solutions satisfy local minimality property,
one need $r>0$ sufficiently small 
such that the measure of $B_r(z)$ small enough yielding some important estimates. 
%
Obviously this simple property heavily depends on the local structure of $\R^n$.
Thus it prevents us
to obtain Rabinowitz and Stredulinsky's locally minimal solutions
for other variational problems,
e.g., 
some problems defined on descrete group $\Z^n$ (please see e.g. \cite{dela, dela1, Mramor, Miao, LC1}).

Naturally one have a question:
can we obtain solutions with the shape of locally minimal solutions by other methods?
For infinite transition solutions, the answer is affirmative
and the present paper is devoted to this problem.
For other types of locally minimal solutions, 
the above question remains open.

The main results of this paper are new proofs of the existence of infinite transition solutions
(please see Theorems \ref{thm:6.81} and \ref{thm:6.82} below).
Our construction of infinite transition solutions also depends on multitransition solutions of \eqref{eq:PDE}.
But different from the proof of \cite[Theorem 8.32]{RS} 
which is relying on local minimality property,
we adopt the first method in \cite[Theorem 6.8]{RS}
with some modifications.
Roughly speaking,
we first establish the existence of $2K$ transition solutions 
and then extract a subsequence converging to an infinite transition solution.
This procedure is natural when one have obtained multitransition solutions.
The advantage of this method  
is we do not need local minimality property,
so it can be used for the cases defined on discrete group. 
However, as pointed out by Rabinowitz and Stredulinsky in \cite[p. 89]{RS},
there are two obstacles in the above process.
The first is to ensure the obtained infinite transition solution has 
the desired asymptotic property,
and the second is how to find a converging subsequence.
We overcome the first difficulty by some comparison results.
For the second, in an intuitive manner,
our method is to copy the ``$2$-transition'' infinite times,
which ensures the existence of a converging subsequence.

The rest of the paper is organized as follows:
Section \ref{sec:2} gives some preliminaries;
in Section \ref{sec:3},
we establish multitransition solutions
by modifying the proof of \cite[Theorem 6.8]{RS};
in Section \ref{sec:4},
we construct infinite transition solutions by multitransition solutions.

\section{Preliminary}\label{sec:2}

Let us begin with Moser's results (cf. \cite{moser, RS}).
We adopt the notations of \cite{RS}.
For $j\in \Z$ and $k\in\{1,2,\cdots, n\}$,
let $\tau_{j}^{k}u(x):=u(x-je_k)$,
where $e_k=(0,\cdots,1,\cdots,0)$,
i.e., the $k$-th coordinate is $1$ and the others $0$.
Set 
$$L(u):=\frac{1}{2}|\nabla u |^2 +F(x,u)$$
and $\T:=\R/\Z$.
Define 
$$\Gamma_0:=W^{1,2}_{loc}(\T^n):=\{u\in W^{1,2}_{loc}(\R^n)\,|\, 
\textrm{$u$ is $1$-periodic in $x_1,\cdots, x_n$}\}.
$$
For $u\in \Gamma_0$, let $$
J_0(u):=\int_{\T^n}L(u)\ud x.$$
Moser showed the following minimal problem 
can be achieved and the minimal points are classical solutions of \eqref{eq:PDE}:
\begin{equation}\label{eq:c0}
\inf_{u\in\Gamma_0}J_0(u).
\end{equation}
In other words, the set 
$$
\MM_0:=\{u\in\Gamma_0 \,|\, J_0(u)=c_0\}$$ 
is not empty 
and consists of solutions of \eqref{eq:PDE},
where $c_0:=\inf_{u\in\Gamma_0}J_0(u)$.
Note that $F$ is $1$-periodic in the last argument $u$,
thus if $u\in\MM_0$, 
then so is $u+k$ ($k\in\Z$). 
Moreover, $\MM_0$ is an ordered set. 
That is, for any $u, v\in\MM_0$,
one and only one of the following holds for all $x\in\R^n$:
$$u\equiv v, \quad  u<v, \quad \textrm{or} \quad u>v.$$ 
Define $\pi_0: C(\R^n)\to \R $ by $\pi_0(u):=u(0)$.
If $\pi_0(\MM_0)=\R$, 
$\MM_0$ is called a foliation of $\R^{n+1}$;
otherwise 
$\MM_0$ is said to be a lamination
and in this case one have the following: 
\begin{equation}\label{eq:*0}
  \textrm{there are adjacent $v_0, w_0\in \MM_0$ with $v_0<w_0$.} \tag{$*_0$}
  \end{equation}
The functions $v$ and $w$ satisfying $v<w$ in some set $\mathcal{B}$ are called adjacent if there does not exist $u\in\mathcal{B}$ such that 
$v\leq u\leq w$.
The lamination case appears frequently and is more complicated since
in this case
there are heteroclinic solutions.
To see this,
assume \eqref{eq:*0} holds.
For $i\in\Z$, set $T_i:= [i, i+1]\times \T^{n-1}$.
Define
\begin{equation*}
  \widehat{\Gamma}_{1}(v_0, w_0):=
  \left\{u \in W_{loc}^{1,2}\left(\mathbb{R} \times \mathbb{T}^{n-1}\right) 
  \mid v_0 \leq u \leq w_0\right\}, \\
\end{equation*}
and
\begin{equation*}
\begin{split}
\Gamma_{1}(v_0, w_0) :=&\{u \in \widehat{\Gamma}_{1}(v_0,w_0) \mid\|u-v_0\|_{L^{2}\left(T_{i}\right)} \rightarrow 0, i \rightarrow-\infty ;\\
&\quad \quad\quad \quad\quad \quad\quad \quad \|u-w_0\|_{L^{2}\left(T_{i}\right)} \rightarrow 0, i \rightarrow \infty\},\\
\Gamma_{1}(w_0, v_0) :=&\{u \in \widehat{\Gamma}_{1}(v_0,w_0) \mid\|u-w_0\|_{L^{2}\left(T_{i}\right)} \rightarrow 0, i \rightarrow-\infty ;\\
&\quad \quad \quad \quad\quad \quad\quad \quad\|u-v_0\|_{L^{2}\left(T_{i}\right)} \rightarrow 0, i \rightarrow \infty\}.
\end{split}
\end{equation*}
For $u\in \widehat{\Gamma}_{1}(v_0, w_0)$ and $p,q\in\Z$ with $p\leq q$,
let 
$$J_{1,p}(u):=\int_{T_p}L(u)\ud x -c_0 \quad \textrm{ and } \quad  J_{1; p,q}(u):= \sum_{i=p}^{q}J_{1,i}(u).$$
In \cite{RS}, it is showed that 
$J_{1; p,q}$ has a lower bound independent of $p,q$.
\begin{lemma}[{\cite[Proposition 2.8]{RS}}]\label{lem:2.8}
There is a constant $\bar{K}_1\geq 0$,
depending on $v_0$ and $w_0$ 
but independent of $p\leq q\in\Z$ and $u\in\widehat{\Gamma}_1(v_0,w_0)$,
such that
$J_{1;p,q}(u)\geq -\bar{K}_1$.
\end{lemma}

By Lemma \ref{lem:2.8},
one can define
$$
J_{1;-\infty, q}(u):= \liminf_{p\to -\infty}J_{1;p,q}(u),
\quad  \quad
J_{1; p, \infty}(u):= \liminf_{q\to \infty}J_{1;p,q}(u)$$
and
$$J_{1}(u):=\liminf_{p\to -\infty \atop q\to \infty}J_{1;p,q}(u).$$

A direct consequence of Lemma \ref{lem:2.8} is:
\begin{lemma}[{\cite[Lemma 2.22]{RS}}]\label{lem:2.22}
If $u\in\widehat{\Gamma}_1(v_0,w_0)$ and $p,q\in\Z$ with $p\leq q$,
\begin{equation*}
  J_{1;p,q}(u)\leq J_1(u)+2\bar{K}_1 \quad \textrm{and} \quad 
  J_{1;p,q}(u)\leq J_{1;-\infty, q}(u)+\bar{K}_1.
\end{equation*}
\end{lemma}
Set 
$$c_1(v_0,w_0):= \inf_{u\in \Gamma_1(v_0,w_0)}J_1(u) 
\quad \textrm{and} \quad
c_1(w_0,v_0):= \inf_{u\in \Gamma_1(w_0,v_0)}J_1(u). $$
By minimization method,
Rabinowitz and Stredulinsky showed the existence of heteroclinic solutions in $x_1$. 
\begin{lemma}[{\cite[Theorem 3.2]{RS}}]\label{lem:3.2}
If $F$ satisfies \eqref{eq:F1}
and \eqref{eq:*0} holds, then
\begin{enumerate}
  \item $\MM_1(v_0,w_0):=\{u\in\Gamma_1(v_0,w_0)\,|\, J_1(u)=c_1(v_0,w_0)\}\neq \emptyset$; %
  \item Any $U\in\MM_1(v_0,w_0)$ satisfies
\begin{enumerate}
  \item $U$ is a classical solution of \eqref{eq:PDE};
  \item $\left\|U-v_{0}\right\|_{C^{2}\left(T_{i}\right)} \rightarrow 0, i \rightarrow-\infty$, \\
$\left\|U-w_{0}\right\|_{C^{2}\left(T_{i}\right)} \rightarrow 0, i \rightarrow \infty$, \\
i.e., $U$ is heteroclinic in $x_1$ from $v_0$ to $w_0$;
  \item $v_0<U<\tau^{1}_{-1}U<w_0$; 
\end{enumerate}
  \item $\MM_1(v_0,w_0)$ is ordered.
\end{enumerate}
Similar conclusion holds for $\MM_1(w_0,v_0):=\{u\in\Gamma_1(w_0,v_0)\,|\, J_1(u)=c_1(w_0,v_0)\}$.
\end{lemma}


In this paper,
for any fixed $K\in \N$,
we first construct $2K$ transition solution, 
which is homoclinic to $v_0$ (please see \eqref{eq:6.6} below).
Then 
we use these $2K$ transition solutions to establish infinite transition solution.

Following Rabinowitz and Stredulinsky \cite{RS},
to construct multi- and infinite transition solution of \eqref{eq:PDE},
one need the following gap conditions:
  \begin{equation}\label{eq:*1}
    \begin{split}
       & \textrm{there are adjacent $v_1, w_1\in \MM_1(v_0,w_0)$ with $v_1<w_1$,}  \\
        \textrm{and }& \textrm{there are adjacent $\bar{v}_1, \bar{w}_1\in \MM_1(w_0,v_0)$ with $\bar{v}_1<\bar{w}_1$.}
    \end{split}\tag{$*_1$}
    \end{equation}
Assume \eqref{eq:*0} and \eqref{eq:*1}.
For $u\in \widehat{\Gamma}_1(v_0,w_0)$, define
\begin{equation*}
\left\{\begin{array}{l}
\rho_{-}(u):=\left\|u-v_{0}\right\|_{L^{2}\left(T_{0}\right)} ,\\
\rho_{+}(u):=\left\|u-w_{0}\right\|_{L^{2}\left(T_{0}\right)}.
\end{array}\right.
\end{equation*}
Set $\bar{\rho}:=\left\|w_{0}-v_{0}\right\|_{L^{2}\left(T_{0}\right)}$.
Choose constants $\rho_{i} \in(0, \bar{\rho}), 1 \leq i \leq 4K$,
such that for $0\leq j\leq K-1$,
\begin{equation}\label{eq:6.2}
\left\{\begin{array}{ll}
\rho_{4j+1} \notin \rho_{-}\left(\mathcal{M}_{1}\left(v_{0}, w_{0}\right)\right), & \rho_{4j+2} \notin \rho_{+}\left(\mathcal{M}_{1}\left(v_{0}, w_{0}\right)\right), \\
\rho_{4j+3} \notin \rho_{+}\left(\mathcal{M}_{1}\left(w_{0}, v_{0}\right)\right), & \rho_{4j+4} \notin \rho_{-}\left(\mathcal{M}_{1}\left(w_{0}, v_{0}\right)\right).
\end{array}\right.
\end{equation}
Let $\ell=(\ell_1,\cdots, \ell_{4K})\in\N^{4K}$ and $m=(m_1,\cdots,m_{4K})\in\Z^{4K}$ with
\begin{equation}\label{eq:6.3}
\begin{split}
 &m_{4j+1}<m_{4j+2}<m_{4j+2}+ \ell_{4j+2} + \ell_{4j+3} \\
 <&m_{4j+3}<m_{4j+4}<m_{4j+4}+\ell_{4j+4}+\ell_{4j+5}<m_{4j+5},
 \end{split}
\end{equation}
and
\begin{equation}\label{eq:apart5}
\textrm{the quantities
appearing in \eqref{eq:6.3} are separated more than $3$.}
\end{equation}
Set
\begin{equation}\label{eq:6.4}
Y_{m, \ell} := Y_{m, \ell}\left(v_{0}, w_{0}\right) 
:=\left\{u \in \widehat{\Gamma}_{1}\left(v_{0}, w_{0}\right) \mid u \text { satisfies }\eqref{eq:6.5}-\eqref{eq:6.6}\right\},
\end{equation}
where
\begin{equation}\label{eq:6.5}
\begin{split}
(i)\,\,& \textrm{$\rho_{-}\left(\tau_{-i}^{1} u\right) \leq \rho_{4j+1}, \quad m_{4j+1}-\ell_{4j+1} \leq i \leq m_{4j+1}-1$},\\
(ii)\,\,& \textrm{$\rho_{+}\left(\tau_{-i}^{1} u\right) \leq \rho_{4j+2}, \quad m_{4j+2} \leq i \leq m_{4j+2}+\ell_{4j+2}-1$},\\
(iii)\,\,& \textrm{$\rho_{+}\left(\tau_{-i}^{1} u\right) \leq \rho_{4j+3}, \quad m_{4j+3}-\ell_{4j+3} \leq i \leq m_{4j+3}-1$},\\
(iv) \,\,& \textrm{$\rho_{-}\left(\tau_{-i}^{1} u\right) \leq \rho_{4j+4}, \quad m_{4j+4} \leq i \leq m_{4j+4}+\ell_{4j+4}-1$},
\end{split}
\end{equation}
with $0\leq j\leq K-1$,
and
\begin{equation}\label{eq:6.6}
\left\|\tau_{-i}^{1} u-v_{0}\right\|_{L^{2}\left(T_{0}\right)} \rightarrow 0, \quad|i| \rightarrow \infty.
\end{equation}
Define
\begin{equation}\label{eq:6.7}
b_{m, \ell} := b_{m, \ell}\left(v_{0}, w_{0}\right) := \inf _{u \in Y_{m, \ell}} J_{1}(u).
\end{equation}

Now the $2K$ transition solution can be obtained by the following theorem.
\begin{theorem}[{\cite[Theorem 8.1]{RS}}]\label{thm:6.8}
Let $F$ satisfy \eqref{eq:F1}.
Assume that \eqref{eq:*0} and \eqref{eq:*1} hold.
Then for sufficiently large $\ell_i \in\N$ ($1\leq i \leq 4K$),
there is a $U=U_{m,\ell}\in Y_{m,\ell}$ such that $J_1(U)=b_{m,\ell}$.
If in addition $m_{4j+2} -m_{4j+1}$ and $m_{4j+4}-m_{4j+3}$ are sufficiently large for all $j=0,\cdots, K-1$,
$U$ is a solution of \eqref{eq:PDE} and
\begin{equation}\label{eq:6.9}
\left\|U-v_{0}\right\|_{C^{2}\left(T_{i}\right)} \rightarrow 0 \quad \text { as }|i| \rightarrow \infty.
\end{equation}
\end{theorem}

In \cite[Theorem 6.8]{RS}, $2$ transition solution is obtained
and
\cite[Theorem 8.1]{RS} is proved as \cite[Theorem 6.8]{RS}.
We also use the same method as in \cite{RS} to prove Theorem \ref{thm:6.8},
but we note some subtle estimates which ensure us obtaining infinite transition solutions.
We postpone the proof of 
Theorem \ref{thm:6.8} until Section \ref{sec:3}. 
The rest of this section will be devoted to other preliminaries.

For the minima $c_1(v_0,w_0)$ and $c_1(w_0,v_0)$, we have:
\begin{lemma}[{\cite[Lemma 6.11]{RS}}]\label{lem:6.11}
$c_1(v_0,w_0)+c_1(w_0,v_0)>0$.
\end{lemma}

For $i\in\Z$, set $X_i:=\cup_{j=-2}^{2}T_{i+j}$.
\begin{lemma}[{\cite[Proposition 6.27]{RS}}]\label{lem:6.27}
Suppose \eqref{eq:*0} holds and $u\in\widehat{\Gamma}_1(v_0,w_0)$ with $J_1(u)\leq M<\infty$.
Then for any $\sigma>0$ and $t\in\Z$,
there is an $\ell_0=\ell_0(\sigma,M)\in\N$ independent of $u$ and $t$ such that whenever $\ell\in\N$ and $\ell\geq \ell_0$,
\begin{equation}\label{eq:6.28}
  \|u-\varphi\|_{L^{2}\left(X_{i}\right)} \leq \sigma
\end{equation}
for some
$i=i(\ell, t) \in(t-\ell+2, t+\ell-2)$ and $\varphi=\varphi_{\ell, t} \in\left\{v_{0}, w_{0}\right\}$.
\end{lemma}
In the proof of Theorem \ref{thm:6.8},
Lemma \ref{lem:6.27} is used in a different way from \cite{RS}.
%
The next lemma provides the required asymptotic result.
\begin{lemma}[{\cite[Proposition 6.53 and Corollary 6.54]{RS}}]\label{lem:6.53}
Suppose \eqref{eq:*0} holds and $u\in\widehat{\Gamma}_1(v_0,w_0)$ with $J_1(u)\leq M<\infty$.
Suppose there is an $R>0$ such that $u$ is a solution of \eqref{eq:PDE} for $x_1\geq R$ (resp. $x_1\le -R$).
Then for some $\varphi\in\{v_0,w_0\}$,
$\|u-\varphi\|_{W^{1,2}\left(X_{i}\right)} \rightarrow 0$
as $i\to\infty$ (resp. $\|u-\varphi\|_{W^{1,2}\left(X_{i}\right)} \rightarrow 0$
as $i\to -\infty$).
Moreover, $\|u-\varphi\|_{C^{2}\left(T_{i}\right)} \rightarrow 0$
as $i\to\infty$ (resp. $\|u-\varphi\|_{C^{2}\left(T_{i}\right)} \rightarrow 0$
as $i\to -\infty$).
\end{lemma}
%
Assume $\rho_i$ ($1\leq i\leq 4K$) satisfy \eqref{eq:6.2}.
For $1\leq j\leq K$,
define
\begin{equation*}
  \begin{split}
&\,\Lambda_{4j}\left(v_{0}, w_{0}\right)\\
:=&\left\{u \in \Gamma_{1}\left(v_{0}, w_{0}\right) \mid\left\|u-v_{0}\right\|_{L^{2}\left(T_{0}\right)}=\rho_{4j+1} \text { or }\left\|u-w_{0}\right\|_{L^{2}\left(T_{0}\right)}=\rho_{4j+2}\right\}
\end{split}
\end{equation*}
and
\begin{equation}\label{eq:6.73}
d_{4j}\left(v_{0}, w_{0}\right):=\inf _{u \in \Lambda_{4j}\left(v_{0}, w_{0}\right)} J_{1}(u).
\end{equation}
$\Lambda_{4j}\left(w_{0}, v_{0}\right)$ and $d_{4j}\left(w_{0}, v_{0}\right)$ are defined similarly.
We have:
\begin{lemma}[{\cite[Proposition 6.74]{RS}}]\label{lem:6.74}
$d_{4j}(v_0,w_0)>c_1(v_0,w_0)$ and $d_{4j}(w_0,v_0)>c_1(w_0,v_0)$.
\end{lemma}

Solutions of \eqref{eq:PDE} have the following apriori estimates:
\begin{lemma}\label{lem:1}
Suppose $U\in \widehat{\Gamma}_1(v_0,w_0)$ satisfies \eqref{eq:PDE} on $[i-1,i+2]\times \T^{n-1}$ for some $i\in\Z$.
Then
\begin{equation}\label{lem:eq1}
  \begin{split}
\|\nabla (U-v_0)\|_{L^2(T_i)} &\leq C\|U-v_0\|_{L^2(\cup_{j=-1}^{1}T_{i+j})},\\
\textrm{and} \quad 
\|\nabla (U-v_0)\|_{L^2(T_i)} &\leq C\|w_0-v_0\|_{L^2(\cup_{j=-1}^{1}T_{i+j})},
  \end{split}
\end{equation}
where $C=C(F)$ is independent of $U$.
\end{lemma}
\begin{proof}
The proof follows from \cite[the proof of (4.71)]{RS} and we omit it here.
\end{proof}
The following lemmas provide some simple but useful estimates.
\begin{lemma}\label{lem:fdf}
  \begin{enumerate}

    \item \label{lem:fdf1}
For $u\in\widehat{\Gamma}_1(v_0,w_0)$,
there is a $C_1=C_1(v_0,w_0, F)>0$ such that
\begin{equation*}
\left|J_{1, p}(u)-\frac{1}{2}\|\nabla(u-v_0)\|_{L^{2}\left(T_{p}\right)}^{2}\right| \leq C_1
\end{equation*}
for all $p\in\Z$.

\item \label{lem:fdf2}
For $u\in\widehat{\Gamma}_1(v_0,w_0)$ and $p \leq x_{1} \leq p+1$,
set
\begin{equation}\label{eq:dddd12dd}
\chi:=\left(x_{1}-p\right) u+\left(p+1-x_{1}\right) v_0
\end{equation}
(or $\chi:=\left(x_{1}-p\right) v_0+\left(p+1-x_{1}\right) u$).
then \begin{equation}
J_{1, p}(\chi) \leq \frac{1}{2}\|\nabla(u-v_0)\|_{L^{2}\left(T_{p}\right)}^{2}+C_1,
\end{equation}
where $C_1$ is defined in Lemma \ref{lem:fdf} \eqref{lem:fdf1}.
Moreover, by Lemma \ref{lem:fdf} \eqref{lem:fdf1} we have
$J_{1,p}(\chi)\leq J_{1,p}(u)+2C_1$.

\item \label{lem:fdf3}
Similar conclusions hold if $v_0$ is replaced by $w_0$ in Lemma \ref{lem:fdf} \eqref{lem:fdf2}.
\end{enumerate}
\end{lemma}

\begin{proof}
For the proof of Lemma \ref{lem:fdf},
please see \cite[the proofs of (2.14) and (2.18)]{RS}.
\end{proof}

\begin{lemma}\label{lem:lemvvv}
  Assume $u\in\widehat{\Gamma}_1(v_0,w_0)$.
  \begin{enumerate}
    \item For any $\gamma>0$, there exists a $\delta=\delta(\gamma)>0$, such that if $\|u-v_0\|_{W^{1,2}(T_i)}\leq \delta$,
     then $|J_{1,i}(u)|< \gamma$. \label{lem:lemvvv1}
     \item Set $\chi:=\left(x_{1}-i\right) u+\left(i+1-x_{1}\right) v_0$
     (or $\chi:=\left(x_{1}-i\right) v_0+\left(i+1-x_{1}\right) u$) on $[i,i+1]\times \T^{n-1}$.
     For any $\gamma>0$, there exists a $\delta=\delta(\gamma)>0$, such that if $\|u-v_0\|_{W^{1,2}(T_i)}\leq \delta$,
     then $|J_{1,i}(\chi)|< \gamma$. \label{lem:lemvvv2}
     \item Similar conclusions hold if $v_0$ is replaced by $w_0$ in 
     Lemma \ref{lem:lemvvv} \eqref{lem:lemvvv1}-\eqref{lem:lemvvv2}. \label{lem:lemvvv3}
\end{enumerate}
\end{lemma}
\begin{proof}
For the proof of Lemma \ref{lem:lemvvv},
please see \cite[the proof of Lemma 2.34]{RS}.
\end{proof}

If $U$ is a solution of \eqref{eq:PDE}, 
by Lemma \ref{lem:1},
Lemma \ref{lem:lemvvv} can be refined as follows.
\begin{lemma}\label{lem:lemvvvu}
  Suppose $U\in \widehat{\Gamma}_1(v_0,w_0)$ satisfies \eqref{eq:PDE} on $[i-1,i+2]\times \T^{n-1}$ for some $i\in\Z$.
  \begin{enumerate}
    \item For any $\gamma>0$, there exists a $\delta=\delta(\gamma)>0$, such that if $\|U-v_0\|_{L^2(\cup_{j=-1}^{1}T_{i+j})}\leq \delta$,
     then $|J_{1,i}(U)|< \gamma$. \label{lem:lemvvvu1}
     \item Set $\chi:=\left(x_{1}-i\right) U+\left(i+1-x_{1}\right) v_0$
     (or $\chi:=\left(x_{1}-i\right) v_0+\left(i+1-x_{1}\right) U$) on $[i,i+1]\times \T^{n-1}$.
     For any $\gamma>0$, there exists a $\delta=\delta(\gamma)>0$, such that if $\|U-v_0\|_{L^2(\cup_{j=-1}^{1}T_{i+j})}\leq \delta$,
     then $|J_{1,i}(\chi)|< \gamma$. \label{lem:lemvvvu2}
     \item Similar conclusions hold if $v_0$ is replaced by $w_0$ 
     in Lemma \ref{lem:lemvvvu} \eqref{lem:lemvvvu1}-\eqref{lem:lemvvvu2}. \label{lem:lemvvvu3}
\end{enumerate}
\end{lemma}

\section{Proof of Theorem \ref{thm:6.8}}\label{sec:3}

In this section,
we prove Theorem \ref{thm:6.8}.
Since the proof is similar to \cite[the Proof of Theorem 6.8]{RS},
we will be brief.
Assume $(u_k)$ is a minimizing sequence for \eqref{eq:6.7}.
Following \cite[the Proof of Theorem 6.8]{RS} with obvious modifications,
we have a $U\in\widehat{\Gamma}_1(v_0,w_0)$  satifying 
\begin{itemize}
  \item $u_k \to U$ in $W^{1,2}_{loc}(\R\times\T^{n-1})$ as $k\to \infty$
and thus \eqref{eq:6.5} holds for $u$ replaced by $U$;
\item 
  $-\bar{K}_1\leq J_{1}(U) \leq M+2 \bar{K}_{1}$
with $\bar{K}_{1}$ defined in Lemma \ref{lem:2.8};
\item $U$ is a solution of \eqref{eq:PDE} except possibly for the $4K$ integral constraint regions.
\end{itemize}
We need to show:
\begin{enumerate}[(A)]

\item for any $i$, $1\leq i\leq 4K$, there exists an 
$\ell_{0,i}:= \ell_{0,i}(\bar{\rho},\rho_i, v_0,w_0,F)\in \N,$
such that 
for $\ell_i\geq \ell_{0,i}$,
there is an $X_{\tilde{j}(i)}$ in the corresponding integral
constraint region such that $U$ satisfies \eqref{eq:PDE} in the interior of $X_{\tilde{j}(i)}$;

\item there exist $\ell_{1,1}:=\ell_{1,1}(\bar{\rho},\rho_{1},v_0,w_0,F)$, 
$\ell_{1,4K}:=\ell_{1,4K}(\bar{\rho},\rho_{4K},v_0,w_0,F)\in \N$,
such that
for $\ell_{1}\geq \ell_{1,1}$ and $\ell_{4K}\geq \ell_{1,4K}$, 
$U$ satisfies \eqref{eq:6.6} and therefore $U\in Y_{m,\ell}$;

\item $J_1(U)=b_{m,\ell}$;
 
\item there exist
$\ell_{2,i}:=\ell_{2,i}(\rho_{i-2},\rho_{i-1},\rho_{i+1},\rho_{i+2},\bar{\rho},\rho_{i},v_0,w_0,F)$ 
($1\leq i\leq 4K$),
$m_{0,j}:=m_{0,j}(\rho_{4j+1},\rho_{4j+2},v_0,w_0,F)$ 
and $m_{1,j}:=m_{1,j}(\rho_{4j+3},\rho_{4j+4},v_0,w_0,F)\in\N$ ($1\leq j\leq K$)
with $\rho_{-1}=\rho_{0}=\rho_{1}$ and $\rho_{4K+1}=\rho_{4K+2}=\rho_{4K+3}=\rho_{4K+4}=\rho_{4K}$,
such that
for $\ell_i\geq \ell_{2,i}$,
$m_{4j+2}-m_{4j+1}\geq m_{0,j}$ 
and $m_{4j+4}-m_{4j+3}\geq m_{1,j}$, %
$U$ satisfies \eqref{eq:PDE} in the integral constraint regions;

\item $U$ satisfies \eqref{eq:6.9}.

\end{enumerate}
\bigskip

\emph{Proof of (A).}
Let $\MMR_i$ be  the $i$-th integral  constraint  regions.
Choose $\sigma=\sigma(\bar{\rho},\rho_i)$ such that
\begin{equation}\label{eq:7.3}
0<\sigma< 
\min\left(\rho_{i}, \bar{\rho}-\rho_{i}\right).
\end{equation}
Set
$M_0:=M_0(v_0,w_0,F):=4C_1+1$ with $C_1=C_1(v_0,w_0,F)$ 
defined in Lemma \ref{lem:fdf}.
We claim that if $\sigma$ satisfies \eqref{eq:7.3},
there is an $\ell_{0,i}:=\ell_{0,i}(\sigma, M_0+2C_1)\in \N$ 
such that whenever $\ell_i\geq \ell_{0,i}$,
\begin{equation}\label{eq:6.28-1}
  \|U-\phi\|_{L^2(X_{\tilde{j}(i)})}\leq \sigma
\end{equation}
for some $X_{\tilde{j}(i)}\subset \MMR_i$ and $\phi=\phi(\ell_i)\in \{v_0,w_0\}$.

We shall prove the claim for $i=4j+2$ and the other cases can be proved similarly.
Firstly we have
\begin{equation}\label{eq:lkdjgd872}
  J_{1;m_{4j+2}-1,m_{4j+2}+\ell_{4j+2}}(U)\leq M_0.
\end{equation}
If
\eqref{eq:lkdjgd872} is false,
then $J_{1;m_{4j+2}-1,m_{4j+2}+\ell_{4j+2}}(U)> M_0,$
so for $k$ large enough, 
\begin{equation}\label{eq:3.200}
J_{1;m_{4j+2}-1,m_{4j+2}+\ell_{4j+2}}(u_k)> M_0.
\end{equation}
Let
\begin{equation*}
  \tilde{u}_k=\left\{
  \begin{array}{ll}
  u_k, & x_1\leq m_{4j+2}-2 \textrm{ or } x_1\geq m_{4j+2}+\ell_{4j+2}+2,  \\
  w_0, & m_{4j+2}-1\leq x_1\leq m_{4j+2}+\ell_{4j+2}+1,\\
  \end{array}
  \right.
\end{equation*}
and interpolate as in \eqref{eq:dddd12dd}.
Hence $\tilde{u}_k\in Y_{m,\ell}$.
Since
\begin{equation*}
  \begin{split}
     J_1(\tilde{u}_k) & = J_{1;-\infty, m_{4j+2}-3}(u_k)+J_{1,m_{4j+2}-2}(\tilde{u}_k)+ J_{1; m_{4j+2}-1,m_{4j+2}+\ell_{4j+2}}(w_0)\\
& \quad +J_{1,m_{4j+2}+\ell_{4j+2}+1}(\tilde{u}_k)+ J_{1; m_{4j+2}+\ell_{4j+2}+2,\infty}(u_k),
  \end{split}
\end{equation*}
we have
\begin{equation*}
  \begin{split}
J_1(\tilde{u}_k)=&J_1(u_k)-J_{1; m_{4j+2}-1,m_{4j+2}+\ell_{4j+2}}(u_k)\\
&\quad +[J_{1,m_{4j+2}-2}(\tilde{u}_k)-J_{1,m_{4j+2}-2}(u_k)]\\
&\quad+[J_{1,m_{4j+2}+\ell_{4j+2}+1}(\tilde{u}_k)-J_{1,m_{4j+2}+\ell_{4j+2}+1}(u_k)].
\end{split}
\end{equation*}
Thus by \eqref{eq:3.200} and Lemma \ref{lem:fdf} \eqref{lem:fdf3},
\begin{equation*}
 b_{m,\ell}\leq  J_1(\tilde{u}_k)< J_1(u_k) - M_0 + 4C_1= J_1(u_k)-1,
\end{equation*}
contradicting $u_k$ is a minimizing sequence for $b_{m,\ell}$.
Hence \eqref{eq:lkdjgd872} holds.

Now set
\begin{equation*}
  \tilde{U}=\left\{
              \begin{array}{ll}
                v_0, & x_1\leq m_{4j+2}-1 \textrm{ or }  x_1\geq m_{4j+2}+\ell_{4j+2}+1, \\
                U, & m_{4j+2}\leq x_1\leq m_{4j+2}+\ell_{4j+2},
              \end{array}
            \right.
\end{equation*}
and interpolate as in \eqref{eq:dddd12dd}.
Then 
by Lemma \ref{lem:fdf} \eqref{lem:fdf2},
$J_1(\tilde{U})\leq M_0+4C_1$.
By Lemma \ref{lem:6.27},
our claim of \eqref{eq:6.28-1}
holds if $U$ is replaced by $\tilde{U}$.
Noting 
\begin{equation}\label{eq:7.4}
  \left\|\tilde{U}-\phi_{i}\right\|_{L^{2}\left(X_{\tilde{j}(i)}\right)} 
=
\left\|U-\phi_{i}\right\|_{L^{2}\left(X_{\tilde{j}(i)}\right)},
\end{equation}
\eqref{eq:6.28-1} holds.

Similar to \cite{RS},
the choice of $\sigma$ in \eqref{eq:7.3} implies
$\phi_{i}=v_{0}$ if 
\begin{equation*}
  \begin{split}
\mathcal{R}_{i}=\mathcal{R}_{4j+1} &:=\left[m_{4j+1}-\ell_{4j+1}, m_{4j+1}\right] \times \mathbb{T}^{n-1},\\ 
\textrm{or } \quad \mathcal{R}_{i}=\mathcal{R}_{4j+4} &:=\left[m_{4j+4}, m_{4j+4}+\ell_{4j+4}\right] \times \mathbb{T}^{n-1},
  \end{split}
\end{equation*}
and
$\phi_{i}=w_{0}$ if 
\begin{equation*}
  \begin{split}
    \mathcal{R}_i=\mathcal{R}_{4j+2} &:=\left[m_{4j+2}-\ell_{4j+2}, m_{4j+2}\right] \times \mathbb{T}^{n-1},\\ 
    \textrm{or }\quad \mathcal{R}_{i}=\mathcal{R}_{4j+3} &:=\left[m_{4j+3}, m_{4j+3}+\ell_{4j+3}\right] \times \mathbb{T}^{n-1}.
  \end{split}
\end{equation*}
The rest of the proof of (A) is same to \cite[p. 82, \textit{Proof of (A)}]{RS},
and we omit it here.

\bigskip

\emph{Proof of (B).}
The proof follows as that of \cite[p. 82, \textit{Proof of (B)}]{RS}.
Noting in this step $\sigma_{4K}$ should be small,
we have to enlarge $\ell_{4K}$.
The lower bound for $\ell_{4K}$, depending on $\sigma_{4K}$, 
will be denoted by $\ell_{1,4K}:=\ell_{1,4K}(\sigma_{4K}, \ell_{0,4K})$.

\bigskip

\emph{Proof of (C).}
Since the proof is same to \cite[p. 84, \textit{Proof of (C)}]{RS},
we omit it here.

\bigskip

\emph{Proof of (D).}
Similar to \cite[p. 84, Proof of (D)]{RS},
we shall only show 
there is strictly inequality with $u=U$ in \eqref{eq:6.5} (i)-(ii).

Suppose by contradiction 
for some $i$ in \eqref{eq:6.5} (i)-(ii) there is equality. Then
\begin{equation}\label{eq:7.24}
\|U-\phi\|_{L^{2}\left(T_{i}\right)}=\rho,
\end{equation}
where $(\phi, \rho)=\left(v_{0}, \rho_{4j+1}\right)$ or
$\left(w_{0}, \rho_{4j+2}\right)$.
Next we assume $0<j<K-1$ and omit the proof of the cases of
$j=0$ and $j=K-1$ (for the proof of these cases, please see \cite[p. 84, \emph{Proof of (D)}]{RS}).
By Lemma \ref{lem:6.27},
there are a 
$\widehat{q}\in [m_{4j}+2, m_{4j}+\ell_{4j}-3]\cap \Z$
and
a 
$q\in [m_{4j+3}-\ell_{4j+3}+2, m_{4j+3}-3]\cap \Z$
such that
\begin{equation*}
\begin{split}
\left\|U-v_{0}\right\|_{L^{2}\left(X_{\widehat{q}}\right)} &\leq  \sigma_{4j},\\
\left\|U-w_{0}\right\|_{L^{2}\left(X_{q}\right)} &\leq  \sigma_{4j+3},
\end{split}
\end{equation*}
where $\sigma_{4j}$, $\sigma_{4j+3}$ will be determined later.
By Lemma \ref{lem:1}, we have
\begin{equation*}
\begin{split}
  & \|\nabla (U-v_{0})\|_{L^{2}\left(\cup_{j=-1}^{1}{T_{\widehat{q}+j}}\right)} \leq
  M_3  \left\|U-v_{0}\right\|_{L^{2}\left(X_{\widehat{q}}\right)} ,\\
&\|\nabla (U-w_{0})\|_{L^{2}\left(\cup_{j=-1}^{1}{T_{q+j}}\right)} \leq
M_3 \left\|U-w_{0}\right\|_{L^{2}\left(X_{q}\right)} ,
\end{split}
\end{equation*}
so
\begin{equation}\label{eq:7.25}
\begin{split}
  &\left\|U-v_{0}\right\|_{W^{1,2}\left(\cup_{j=-1}^{1}{T_{\widehat{q}+j}}\right)} \leq (M_3 +1) \sigma_{4j},\\
  &\left\|U-w_{0}\right\|_{W^{1,2}\left(\cup_{j=-1}^{1}{T_{q+j}}\right)} \leq (M_3 +1) \sigma_{4j+3}.
  \end{split}
\end{equation}
Define $U^*$ via
\begin{equation}\label{eq:7.26}
U^{*}=\left\{\begin{array}{ll}
U, & x_{1} \leq \widehat{q}-1, \\
v_{0}, & \widehat{q} \leq x_{1} \leq \widehat{q}+1, \\
U, & \widehat{q}+2 \leq x_{1} \leq q-1 ,\\
w_{0}, & q \leq x_{1} \leq q+1, \\
U, & q+2 \leq x_{1},
\end{array}\right.
\end{equation}
and interpolate as in \eqref{eq:dddd12dd}.
Then 
by \eqref{eq:7.25} and \eqref{eq:7.26},
there is a function $\kappa(\theta,\theta')$ with $\kappa(\theta,\theta')\to 0$ as $\theta, \theta'\to 0$ such that
\begin{equation}\label{eq:7.27}
\left|J_{1}(U)-J_{1}\left(U^{*}\right)\right| \leq \kappa(\sigma_{4j},\sigma_{4j+3}).
\end{equation}

Define
\begin{equation*}
\Theta=\left\{\begin{array}{ll}
U^{*}, &  x_{1}\leq \widehat{q},\\
v_{0}, & \widehat{q}\leq x_{1} , \\
\end{array}\right.
\end{equation*}
\begin{equation*}
\Phi=\left\{\begin{array}{ll}
v_{0}, & x_{1}\leq \widehat{q},\\
U^{*}, & \widehat{q}\leq x_{1} \leq q+1, \\
w_{0}, & q+1 \leq x_{1},
\end{array}\right.
\end{equation*}
and
\begin{equation*}
\Psi=\left\{\begin{array}{ll}
w_{0} & x_{1} \leq q, \\
U^{*}, & q \leq x_{1}.
\end{array}\right.
\end{equation*}
Note that
$\tau_{i}^{1} \Phi \in \Lambda_{4j}\left(v_{0}, w_{0}\right)$.
Therefore by Lemma \ref{lem:6.74},
\begin{equation}\label{eq:7.30}
J_{1}(\Phi)=J_{1}\left(\tau_{i}^{1} \Phi\right) \geq d_{4j}\left(v_{0}, w_{0}\right).
\end{equation}
Since $\Psi \in Y_{m',\ell'}$ with $m'=(m_{4j+3}, \cdots, m_{4K})$ and $\ell'=(\ell_{4j+3}, \cdots, \ell_{4K})$
and obvious modification of \eqref{eq:6.6},
\begin{equation}\label{eq:7.31}
J_{1}(\Psi) \geq  b_{m',\ell'}.  
\end{equation}
Similarly
\begin{equation}\label{eq:7.31-1}
J_{1}(\Theta) \geq  b_{\widehat{m},\widehat{\ell}},  
\end{equation}
with $\widehat{m}=(m_{1}, \cdots, m_{4j})$ and $\widehat{\ell}=(\ell_{1}, \cdots, \ell_{4j})$.
Observing that
\begin{equation}\label{eq:7.32}
J_{1}\left(U^{*}\right)=J_{1}(\Theta)+J_{1}(\Phi)+J_{1}(\Psi),
\end{equation}
by \eqref{eq:7.27}-\eqref{eq:7.32} we have:
\begin{equation}\label{eq:7.33}
J_{1}(U) \geq b_{\widehat{m},\widehat{\ell}}+ d_{4j}\left(v_{0}, w_{0}\right)+b_{m',\ell'}-\kappa(\sigma_{4j},\sigma_{4j+3}).
\end{equation}

On the other hand, 
we can estimate $J_1(U)$ to obtain 
an upper bound as in \cite{RS}.
Indeed, let 
\begin{equation}\label{eq:ch.eps}
  \epsilon := \frac{1}{2}\left[d_{4j}(v_0,w_0)-c_1(v_0,w_0)\right].
\end{equation}
By Lemma \ref{lem:6.74}, $\epsilon>0$.
Since $d_{4j}(v_0,w_0)$ depends on $\rho_{4j+1},\rho_{4j+2},v_0,w_0$,
so is $\epsilon$. 
For this fixed $\epsilon$,
there exists $m_{0,j}=m_{0,j}(\epsilon)\in\N$,
such that if $m_{4j+2}-m_{4j+1}\geq m_{0,j}$
then 
we can find $V_{1} \in \mathcal{M}_1\left(v_{0}, w_{0}\right)$,
$U_1 ^{\epsilon}\in Y_{\widehat{m},\widehat{\ell}}$
and $U_2 ^{\epsilon}\in Y_{m',\ell'}$,
($U_1 ^{\epsilon}$ and $U_2 ^{\epsilon}$ need not to be 
minimizers of $J_1$ on $Y_{\widehat{m},\widehat{\ell}}$ and $Y_{m',\ell'}$)
satisfying if
\begin{equation}\label{eq:7.34}
\widehat{U}=\left\{\begin{array}{ll}
U_1^{\epsilon}, & x_{1} \leq \widehat{q}-1, \\
v_{0}, & \widehat{q} \leq x_{1} \leq \widehat{q}+1 ,\\
V_{1}, & \widehat{q}+2\leq x_{1} \leq q-1, \\
w_{0}, & q \leq x_{1} \leq q+1 ,\\
U_2^{\epsilon}, & q+2 \leq x_{1},
\end{array}\right.
\end{equation}
then
\begin{equation}\label{eq:7.35}
J_{1}(U) \leq J_{1}(\widehat{U}) \leq b_{\widehat{m},\widehat{\ell}}+c_{1}\left(v_{0}, w_{0}\right)+b_{m',\ell'}+\epsilon.
\end{equation}
By \eqref{eq:7.33}-\eqref{eq:7.35},
\begin{equation}\label{eq:7.36}
\epsilon=\frac{1}{2}\left[d_{4j}\left(v_{0}, w_{0}\right)-c_{1}\left(v_{0}, w_{0}\right) \right]\leq \kappa(\sigma_{4j},\sigma_{4j+3}).
\end{equation}
Finally choosing 
$\sigma_{4j}=\sigma'' _{4j}(\epsilon):=\sigma'' _{4j}(\rho_{4j+1},\rho_{4j+2},v_0,w_0)$ 
and $\sigma_{4j+3}=\sigma'' _{4j+3}(\epsilon):=\sigma'' _{4j+3}(\rho_{4j+1},\rho_{4j+2},v_0,w_0)$ 
so small---of course, $\ell_{4j}$, $\ell_{4j+3}$ 
should be large enough, for example satisfying
$$\ell_{4j}\geq \ell_{2,4j}:=\ell_{2,4j}(\rho_{4j+1},\rho_{4j+2},\bar{\rho},\rho_{4j},v_0,w_0,F):=\ell_{2,4j}(\sigma'' _{4j},\ell_{0,4j}),$$
$$\ell_{4j+3}\geq \ell_{2,4j+3}:=\ell_{2,4j+3}(\rho_{4j+1},\rho_{4j+2},\bar{\rho},\rho_{4j+3},v_0,w_0,F):=\ell_{2,4j+3}(\sigma'' _{4j+3},\ell_{0,4j+3})$$
---such that
\begin{equation}\label{eq:7.37}
\kappa(\sigma_{4j},\sigma_{4j+3})< \epsilon 
\end{equation}
holds. 
But \eqref{eq:7.36} and \eqref{eq:7.37} are not compatible. 
Thus we have a contradiction and (D) is proved.

\bigskip

\emph{Proof of (E).}
Since the proof is same to \cite[p. 86, \textit{Proof of (E)}]{RS},
we omit it here.

\bigskip

Thus we complete the proof of Theorem \ref{thm:6.8}.

\section{Infinite transition solutions}\label{sec:4}

Let us turn to the construction of infinite transition solutions.
We follow \cite{RS} to construct three cases of infinite transition solutions (cf. \cite[p. 89]{RS}):
\begin{enumerate}[(i)]
  \item $m=(m_k)_{k\in\N}$ with $m_k\to \infty$ as $k\to \infty$; \label{enum:1}
  \item $m=(m_k)_{k\in-\N}$ with $m_k\to -\infty$ as $k\to -\infty$; \label{enum:2}
  \item $m=(m_k)_{k\in\Z}$ with $m_k\to \pm\infty$ as $k\to \pm\infty$. \label{enum:3} 
\end{enumerate}
As pointed out by Rabinowitz and Stredulinsky (\cite[pp. 89-90]{RS}),
when one want to construct infinite transition solutions by \cite[Theorem 8.1]{RS},
there are two possible obstacles one have to face:
\begin{enumerate}[(D-a)]
  \item 
The lesser one is to show that for cases \eqref{enum:1} and \eqref{enum:2},
the infinite transition solution 
$U^*$ has the appropriate asymptotic behavior. \label{diff:a}
\item 
A more serious difficulty is in applying \cite[Theorem 8.1]{RS} to find $u^* _j$.
That result requires
$\ell, m_{2i}-m_{2i -1} \gg 0$
and a priori $\ell$ and the difference in the $m_i$’s will depend on $j$
and possibly go to $\infty$ as $j\to \infty$. \label{diff:b}
\end{enumerate}
For the easier case of type \eqref{enum:3},
(D-\ref{diff:a}) does not appear.
In our proof of Theorem \ref{thm:6.8},
only Step (D) may lead to (D-\ref{diff:b})
for an infinite sequence $\{\rho_i\}_{i\in\Z}$.
To overcome this difficulty,
we set 
\begin{equation}\label{numbera}
  \#\{\rho_i\}_{i\in\Z}<\infty,
\end{equation}
where $\# A$ is the cardinality of the set $A$.
Then $\ell_{2,i}, m_{0,i}, m_{1,i}$ depend on finitely many $\rho_i$'s
and thus cannot go to $\infty$ as $i\to \infty$.
With the above observation,
we have

\begin{theorem}\label{thm:6.81}
  Let $F$ satisfy \eqref{eq:F1}.
  Assume that \eqref{eq:*0} and \eqref{eq:*1} holds.
  Let $m=(m_i)_{i\in\Z}$ and $\ell=(\ell_i)_{i\in\Z}$.
  Then  for sufficiently  large $\ell_i\in\N$ and $m_{i+1}-m_i \in\N$,
  there is a solution $U$ of \eqref{eq:PDE} in $\widehat{\Gamma}(v_0,w_0)$ experiencing infinite transitions.
  \end{theorem}
For the proof of Theorem \ref{thm:6.81},
we refer the reader to \cite[the proof of Theorem 8.32]{RS}.
\begin{remark}
  In Section \ref{sec:1},
  we say (D-\ref{diff:b}) is overcomed by copying the $2$-transition solution infinite times.
  In fact, this is a special case of \eqref{numbera}.
  By changing $(\rho_i)_{i\in\Z}$ in \eqref{numbera},
  we obtain uncountable number of infinite transition solutions.
    \end{remark}
Next we consider infinite transition solutions of type \eqref{enum:1}.
Type \eqref{enum:2} can be treated similarly.
Before that we make a remark on the proof of Theorem \ref{thm:6.8}.
\begin{remark}\label{rem:d0}
  To construct infinite transition solution of type \eqref{enum:1}, 
  we need another restriction on $\sigma=\sigma(\bar{\rho},\rho_1)$ 
  (and thus $\ell_{0,1}$),
  which appears in the proof of Theorem \ref{thm:6.8} (the proof of (A)).
For $\gamma:= [c_1(v_0,w_0)+c_1(w_0,v_0)]/6>0$,
by Lemma \ref{lem:lemvvvu},
there exists a $\delta=\delta(\gamma)>0$,
such that if $u$ is a solution of \eqref{eq:PDE} on $[i-1,i+2]\times \T^{n-1}$ satisfying $\|u-v_0\|_{L^2(\cup_{j=-1}^{1}T_{i+j})}<\delta$,
then 
\begin{equation}
  |J_{1,i}(u)|<\gamma.
\end{equation}
We ask $\sigma=\sigma(\bar{\rho},\rho_1)<\delta(\gamma)$ for all $K$,
thus the solutions $U_{2K}$ obtaining by Theorem \ref{thm:6.8} satisfy
\begin{equation}\label{eq:u2k}
  |J_{1,\tilde{j}(1)}(U_{2K})|<\gamma,
\end{equation}
where $\tilde{j}(1)$ is obtained in the proof of (A) in Theorem \ref{thm:6.8}
and $\tilde{j}(1)$ may depend on $K$.

Moreover, set 
  \begin{equation*}
    f_1:= \left\{
      \begin{array}{ll}
        U_{2K}, & x_1\leq \tilde{j}(1), \\
        v_0, & x_1\geq \tilde{j}(1)+1,
      \end{array}
    \right.
    \end{equation*}
    and
    \begin{equation*}
      f_2:= \left\{
        \begin{array}{ll}
          v_0, & x_1\leq \tilde{j}(1), \\
          U_{2K}, & x_1\geq \tilde{j}(1)+1,
        \end{array}
      \right.
      \end{equation*}
      and interpolate as in \eqref{eq:dddd12dd}.
By Lemma \ref{lem:lemvvvu},
we have 
\begin{equation}\label{eq:f1f2}
  |J_{1,\tilde{j}(1)}(f_1)|<\gamma \quad \textrm{and} \quad|J_{1,\tilde{j}(1)}(f_2)|<\gamma.
\end{equation}
  \end{remark}

  \begin{theorem}\label{thm:6.82}
    Let $F$ satisfy \eqref{eq:F1}.
    Assume that \eqref{eq:*0} and \eqref{eq:*1} holds.
    Let $m=(m_i)_{i\in\N}$ and $\ell=(\ell_i)_{i\in\N}$.
    Then  for sufficiently  large $\ell_i\in\N$ and $m_{i+1}-m_i \in\N$,
    there is a solution $U$ of \eqref{eq:PDE} experiencing infinite transitions and
    \begin{equation}\label{eq:6.922}
    \left\|U-v_{0}\right\|_{C^{2}\left(T_{i}\right)} \rightarrow 0 \quad \text { as }i \rightarrow -\infty.
    \end{equation}
    \end{theorem}

\begin{proof}
It suffices to show \eqref{eq:6.922}.
Take $(\rho_i)_{i\in\N}$ satisfying \eqref{eq:6.2} and \eqref{numbera}.
Then we can select $\ell_i\in\N$ and $m_{i+1}-m_i \in\N$ large enough,
such that for any $K>1$, 
$$J_1(U_{2K})=\inf_{u\in Y_{\bar{m}, \bar{\ell}}}J_1(u)$$
with $\bar{m}=\{m_i\}_{i=1}^{4K}$ and $\bar{\ell}=\{\ell_i\}_{i=1}^{4K}$.
The limit $U$ of $U_{2K}$ (maybe up to a subsequence) satisfies
$\rho_{-}\left(\tau_{-i}^{1} U\right) \leq \rho_{1}$, 
$ i \leq m_{1}-1$.
Similar proof of \eqref{eq:lkdjgd872} shows that
$J_{1;-\infty ,m_1}(U_{2K})\leq \widetilde{M}_0$
for some $\widetilde{M}_0>0$ independent of $K$.
So for $q<m_1$,
by Lemma \ref{lem:2.22},
\begin{equation*}
  J_{1;q,m_1}(U)=\lim_{K\to \infty}J_{1;q,m_1}(U_{2K})
  \leq \lim_{K\to \infty}[J_{1;-\infty,m_1}(U_{2K})+ \bar{K}_1]
  \leq \widetilde{M}_0+\bar{K}_1.
\end{equation*}
Hence $J_{1;-\infty,m_1}(U)\leq \widetilde{M}_0+\bar{K}_1$.
Set
\begin{equation*}
  \tilde{U}=\left\{
              \begin{array}{ll}
                U, & x_1\leq m_1, \\
                v_0, & x_1\geq m_1+1,
              \end{array}
            \right.
\end{equation*}
and interpolate as in \eqref{eq:dddd12dd}.
Then $J_1(\tilde{U})<\infty$.
By Lemma \ref{lem:6.53},
one obtain
\begin{equation}\label{eq:ldld}
  \|U-\phi\|_{C^{2}\left(T_{i}\right)} = 
\|\tilde{U}-\phi\|_{C^{2}\left(T_{i}\right)} \rightarrow 0 \quad \text { as }i \rightarrow - \infty
\end{equation}
with $\phi=v_0$ or $w_0$.
If $\phi = v_0$, 
\eqref{eq:6.922} holds.
Next we shall show $\phi=w_0$ is impossible
and thus complete the proof of Theorem \ref{thm:6.82}.

Suppose by contradiction $\phi=w_0$.
By \eqref{eq:ldld},
for any $\sigma>0$,
there exists a $q\in\Z$ with $q<m_1 - \ell_1 -6$, such that
$\| U- w_0 \|_{C^2(X_q)}<\sigma$.
Then for $K$ large enough,
we have $\| U_{2K} - w_0 \|_{C^2(X_q)}<\sigma$.
%

Set 
\begin{equation*}
f_1:= \left\{
  \begin{array}{ll}
    U_{2K}, & x_1\leq \tilde{j}(1), \\
    v_0, & x_1\geq \tilde{j}(1)+1,
  \end{array}
\right.
\end{equation*}
and
\begin{equation*}
  f_2:= \left\{
    \begin{array}{ll}
      v_0, & x_1\leq \tilde{j}(1), \\
      U_{2K}, & x_1\geq \tilde{j}(1)+1,
    \end{array}
  \right.
  \end{equation*}
  and interpolate as in \eqref{eq:dddd12dd}.
  Thus $f_2\in Y_{\bar{m},\bar{\ell}}$ with $\bar{m}:= \{m_i\}_{i=1}^{4K}$ and $\bar{\ell}:=\{\ell_i\}_{i=1}^{4K}$,
  so 
  \begin{equation*}
  J_1(f_2)\geq b_{\bar{m},\bar{\ell}}.
  \end{equation*}
  By the definition of $U_{2K}$,
  \begin{equation*}
    J_1(U_{2K}) = b_{\bar{m},\bar{\ell}}.
    \end{equation*}
Then
\begin{equation}\label{eq:j1u2k}
  \begin{split}
J_1(U_{2K}) = &J_{1; -\infty, \tilde{j}(1)-1}(U_{2K}) + J_{1; \tilde{j}(1), \infty}(U_{2K})\\
=& J_{1;-\infty, \tilde{j}(1)-1}(f_1) + J_{1,\tilde{j}(1)}(U_{2K})+ J_{1; \tilde{j}(1)+1, \infty}(f_2)\\
=& [J_{1}(f_1)-J_{1,\tilde{j}(1)}(f_1)] + J_{1,\tilde{j}(1)}(U_{2K}) +[J_1(f_2)-J_{1,\tilde{j}(1)}(f_2)]\\
\geq & J_{1}(f_1) + b_{\bar{m},\bar{\ell}} -|J_{1,\tilde{j}(1)}(f_1)| - |J_{1,\tilde{j}(1)}(U_{2K})|- |J_{1,\tilde{j}(1)}(f_2)|.
  \end{split}
\end{equation}
By \eqref{eq:u2k} and \eqref{eq:f1f2}, 
for $\sigma \in (0, \delta(\gamma))$ as in Remark \ref{rem:d0},
we have
\begin{equation}\label{eq:3ineq}
|J_{1,\tilde{j}(1)}(f_1)|<\gamma, \quad |J_{1,\tilde{j}(1)}(U_{2K})|<\gamma \quad
\textrm{and}\quad |J_{1,\tilde{j}(1)}(f_2)|<\gamma,
\end{equation}
where $\gamma := [c_1(v_0,w_0)+ c_1(w_0,v_0)]/6>0$.
Thus by \eqref{eq:j1u2k},
\begin{equation}\label{eq:bj1ineq}
  b_{\bar{m},\bar{\ell}}> J_{1}(f_1) + b_{\bar{m},\bar{\ell}} - \frac{c_1(v_0,w_0)+ c_1(w_0,v_0)}{2}.
\end{equation}
If 
\begin{equation}\label{eq:j1ineq}
  J_{1}(f_1) \geq \frac{c_1(v_0,w_0)+ c_1(w_0,v_0)}{2},
\end{equation} 
then we obtain a contradiction and complete the proof of Theorem \ref{thm:6.82}.

To this end,
set
\begin{equation*}
  g_1:= \left\{
    \begin{array}{ll}
      f_{1}, & x_1\leq q, \\
      w_0, & x_1\geq q+1,
    \end{array}
  \right.
  \end{equation*}
and
\begin{equation*}
  g_2:= \left\{
    \begin{array}{ll}
      w_0, & x_1\leq q, \\
      f_{1}, & x_1\geq q+1,
    \end{array}
  \right.
  \end{equation*}
  and interpolate as in \eqref{eq:dddd12dd}.
  Thus $g_1\in \Gamma_1(v_0,w_0)$ and $g_2\in \Gamma_1(w_0,v_0)$,
  so $J_1(g_1)\geq c_1(v_0,w_0)$ and $J_1(g_2)\geq c_1(w_0,v_0)$.
Then 
\begin{equation*}
\begin{split}
J_1(f_1) = &J_{1; -\infty, q-1}(f_1) + J_{1,q}(f_1) + J_{1;q+1,\infty}(f_1)\\
=&[J_1(g_1)-J_{1,q}(g_1)]+ J_{1,q}(f_1) + [J_1(g_2)-J_{1,q}(g_2)]\\
\geq &c_1(v_0,w_0) + c_1(w_0,v_0) - |J_{1,q}(f_1)|-|J_{1,q}(g_1)|-|J_{1,q}(g_2)|\\
\geq &\frac{c_1(v_0,w_0)+c_1(w_0,v_0)}{2},
\end{split}
\end{equation*}
where the last inequality is obtained  
similar to \eqref{eq:3ineq}-\eqref{eq:bj1ineq}.
Hence \eqref{eq:j1ineq} holds and \eqref{eq:bj1ineq} leads to a contradiction, 
which completes the proof of Theorem \ref{thm:6.82}.
\end{proof}

Let us conclude this section by
showing that infinite transition solutions obtained by Theorems
\ref{thm:6.81} and \ref{thm:6.82}
are locally minimal\footnote{We thank Professor Zhi-Qiang Wang for suggesting this proposition.}.
We emphasize our construction of infinite trasition solutions
does not depend on local minimality property.

\begin{proposition}\label{prop:local}
Assume $U$ is an infinite transition solution obtained by 
Theorem \ref{thm:6.81} or Theorem \ref{thm:6.82}.
Then $U$ is locally minimal.
\end{proposition}

The proof of Proposition \ref{prop:local} depends on 
the following lemma.
Set $$\mathcal{M}(Y_{m,\ell}(v_0,w_0)):=\{u\in Y_{m,\ell} \,|\, J_1(u)=b_{m,\ell}\}.$$
\begin{lemma}[{\cite[Proposition 8.12]{RS}}]\label{prop:8.12}
  Any $V\in \mathcal{M}(Y_{m,\ell}(v_0,w_0))$ 
  is locally minimal in $B_r(z)$ 
  for any $z\in \R \times \T^{n-1}$ and small $r > 0$.
\end{lemma}

A careful reading of the proof of Lemma \ref{prop:8.12} in \cite{RS} shows
there only involves $V$ minimizes $J_1$ over $Y_{m,\ell}$.
Thus Lemma \ref{prop:8.12} holds for the multitransition solutions of Theorem \ref{thm:6.8}.

\bigskip
\noindent\emph{Proof of Proposition \ref{prop:local}.}
For any $z\in \R \times \T^{n-1}$ and small $r > 0$ with $B_r(z)\subset \R\times \T^{n-1}$, 
we need to show $U$ 
  minimizes
  $$I_{r,z} =\int_{B_r(z)}L(u) \mathrm{d} x$$
  over
  $$
  E^{\infty}_{r,z} := \{u\in W^{1,2}_{loc}(\R\times \T^{n-1}) \,|\, u=U \text{ on }  (\R\times \T^{n-1})\setminus B_r (z)\}.
$$

Fix any $u\in E^{\infty}_{r,z}$.
Choose $r' >r$ with $r'$ small.
Select $V_k$ of Theorem \ref{thm:6.8} such that $V_k \to U$ in $C^2_{loc}(\R\times \T^{n-1})$ as $k\to \infty$.
By Lemma \ref{prop:8.12}, 
we have 
\begin{equation}\label{eq:I}
  I_{r',z}(\tilde{u}) \geq I_{r',z}(V_k)
\end{equation}
for any $\tilde{u} \in E^{k}_{r',z}:= \{u\in W^{1,2}_{loc}(\R\times \T^{n-1}) \,|\, u=V_k \text{ on }  (\R\times \T^{n-1})\setminus B_{r'} (z)\}$.
Define
\begin{equation*}
v_k = \left\{
  \begin{aligned}
    u & , & x\in B_r(z), \\
    \frac{\sqrt{x_1 ^2 + \cdots +x_n ^2} -r}{r' - r}V_k + \frac{r' - \sqrt{x_1 ^2 + \cdots +x_n ^2}}{r' - r}u & , & x\in B_{r'}(z)\setminus B_r(z), \\
    V_k & , & x\in \R\times \T^{n-1}\setminus B_{r'}(z).
    \end{aligned}
  \right.
\end{equation*}
Then $v_k\in E^{k}_{r',z}$.
By \eqref{eq:I},
\begin{equation}\label{eq:I1}
  I_{r',z}(v_k) \geq I_{r',z}(V_k).
\end{equation}
Notice $V_k \to U$ in $C^2_{loc}(\R\times \T^{n-1})$ as $k\to \infty$.
Letting $k\to \infty$ in \eqref{eq:I1},
we have
\begin{equation*}
  I_{r',z}(u) \geq I_{r',z}(U).
\end{equation*}
But $u=U$ on $B_{r'}(z)\setminus B_r(z)$,
so
\begin{equation*}
  I_{r,z}(u) \geq I_{r,z}(U),
\end{equation*}
completing the proof of Proposition \ref{prop:local}.
\qed


\subsection*{Acknowledgments}
The author is greatly indebted to
Professor Zhi-Qiang Wang (Utah State University) 
for his active interest, valuable encouragements and helpful discussions 
during the preparation of the paper. 
The author is supported by NSFC: 12201162.



\end{document}